\documentclass[12pt]{article}
\usepackage{amsmath,amssymb,amsthm, comment}

\newtheorem{theorem}{Theorem}[section]
\newtheorem{lemma}[theorem]{Lemma}
\newtheorem{corollary}[theorem]{Corollary}

\newtheorem{proposition}[theorem]{Proposition}

\newtheoremstyle{named}{}{}{\itshape}{}{\bfseries}{.}{.5em}{\thmnote{#3's }#1}
\theoremstyle{named}

\newtheoremstyle{nnamed}{}{}{\itshape}{}{\bfseries}{.}{.5em}{\thmnote{#3' }#1}
\theoremstyle{nnamed}

\newcommand{\calc}{{\cal C}}

\def\barr{\begin{array}}
\def\earr{\end{array}}

\title{Finite groups with two Chermak-Delgado measures}
\author{Marius T\u arn\u auceanu}
\date{November 19, 2018}

\begin{document}

\maketitle

\begin{abstract}
In this note, we study the finite groups whose Chermak-Delgado measure has exactly two values.
They determine an interesting class of $p$-groups containing cyclic groups of prime order and extraspecial $p$-groups.
\end{abstract}

{\small
\noindent
{\bf MSC2000\,:} Primary 20D30; Secondary 20D60, 20D99.

\noindent
{\bf Key words\,:} Chermak-Delgado measure, Chermak-Delgado lattice, subgroup lattice, generalized quaternion $2$-group, extraspecial $p$-group, outer abelian $p$-group, $p$-group of maximal class.}

\section{Introduction}

Throughout this paper, let $G$ be a finite group and $L(G)$ be the subgroup lattice of $G$. Denote by
\begin{equation}
m_G(H)=|H||C_G(H)|\nonumber
\end{equation}the \textit{Chermak-Delgado measure} of a subgroup $H$ of $G$ and let
\begin{equation}
m^*(G)={\rm max}\{m_G(H)\mid H\leq G\} \mbox{ and } {\cal CD}(G)=\{H\leq G\mid m_G(H)=m^*(G)\}.\nonumber
\end{equation} Then the set ${\cal CD}(G)$ forms a modular, self-dual sublattice of $L(G)$, which is called the \textit{Chermak-Delgado lattice} of $G$. It was first introduced by Chermak and Delgado \cite{7}, and revisited by Isaacs \cite{9}. In the last years there has been a growing interest in understanding this lattice (see e.g. \cite{3,4,5,6,8,10,11,12,13,15,18,20}). We recall several important properties of the Chermak-Delgado measure that will be used in our paper:
\begin{itemize}
\item[$\cdot$] if $H\leq G$ then $m_G(H)\leq m_G(C_G(H))$, and if the measures are equal then $C_G(C_G(H))=H$;
\item[$\cdot$] if $H,K\leq G$ then $m_G(H)m_G(K)\leq m_G(\langle H,K\rangle)m_G(H\cap K)$, and the equality occurs if and only if $\langle H,K\rangle=HK$ and $C_G(H\cap K)=C_G(H)C_G(K)$;
\item[$\cdot$] if $H\in {\cal CD}(G)$ then $C_G(H)\in {\cal CD}(G)$ and $C_G(C_G(H))=H$;
\item[$\cdot$] the minimum subgroup $M(G)$ of ${\cal CD}(G)$ (called the \textit{Chermak-Delgado subgroup} of $G$) is characteristic, abelian, and contains $Z(G)$.
\end{itemize}

We remark that the Chermak-Delgado measure associated to a finite group $G$ can be seen as a function $$m_G:L(G)\longrightarrow\mathbb{N}^*,\, H\mapsto m_G(H),\, \forall\, H\in L(G).$$The starting point for our discussion is given by Corollary 3 of \cite{16}, which states that there is no finite non-trivial group $G$ such that ${\cal CD}(G)=L(G)$. In other words, $m_G$ has at least two distinct values for every finite non-trivial group $G$. This leads to the following natural question:

\bigskip
\noindent{\it \hspace{5mm}Which are the finite groups $G$ whose Chermak-Delgado measure $m_G$ has exactly two values}?
\bigskip

In what follows, let $\calc$ be the class of finite groups satisfying the above property. Its study is the main goal of the current note.
\bigskip

We recall several basic definitions:
\begin{itemize}
\item[-] a \textit{generalized quaternion $2$-group} is a group of order $2^n$, $n\geq 3$, defined by the presentation $$Q_{2^n}=\langle a,\,b\mid a^{2^{n-1}}=1,\, a^{2^{n-2}}=b^2,\, b^{-1}ab=a^{-1}\rangle;$$
\item[-] a finite $p$-group $G$ is said to be \textit{extraspecial} if $Z(G)=G'=\Phi(G)$ has order $p$;
\item[-] a finite $p$-group $G$ is said to be \textit{outer abelian} if $G$ is non-abelian, but every proper quotient group of $G$ is abelian;
\item[-] a finite $p$-group $G$ of order $p^n$ is said to be \textit{of maximal class} if the nilpotence class of $G$ is $n-1$.
\end{itemize}

The following results on $p$-groups will be useful to us.  Lemmas 1.1 and 1.2 appear in (4.26) and (4.4) of \cite{14}, II, Lemma 1.3 in Corollary 10 of \cite{19}, and Lemma 1.4 in Proposition 1.8 of \cite{1}.

\begin{lemma}
Any group of order $p^4$ contains an abelian subgroup of order $p^3$.
\end{lemma}

\begin{lemma}
A finite $p$-group $G$ has a unique subgroup of order $p$ if and only if either it is cyclic or $p=2$ and $G\cong Q_{2^n}$ for some $n\geq 3$.
\end{lemma}

\begin{lemma}
A finite $p$-group $G$ is outer abelian if and only if $|G'|=p$ and $Z(G)$ is cyclic, and $G$ is one of the following non-isomorphic
groups:
\begin{itemize}
\item[{\rm a)}] $M(n,1)=\langle a,b\mid a^{p^n}=b^p=1, a^b=a^{1+p^{n-1}}\rangle$, $n\geq 3$;
\item[{\rm b)}] an extraspecial $p$-group;
\item[{\rm c)}] $G=E*A$, where $E$ is an extraspecial $p$-group and $A\cong M(n,1)$, $n\geq 3$;
\item[{\rm d)}] $G=E*A$, where $E$ is an extraspecial $p$-group and $A\cong C_{p^t}$, $t\geq 2$
\end{itemize}
\end{lemma}

\begin{lemma}
A finite $p$-group $G$ is of maximal class if and only if it has a subgroup $A$ of order $p^2$ such that $C_G(A)=A$.
\end{lemma}

\section{Main results}

Our first result indicates an important property of the groups in $\calc$.

\begin{theorem}
If a finite group $G$ is contained in $\calc$, then $|Z(G)|$ is a prime.
\end{theorem}

\begin{proof}
Assume that $|Z(G)|$ is not a prime.

If $|Z(G)|=1$, then we have $m_G(1)=m_G(G)=|G|$. Also, $G$ cannot be a $p$-group. It follows that there are at least two distinct primes $p$ and $q$ dividing $|G|$. Let $S_p$ and $S_q$ be a Sylow $p$-subgroup and a Sylow $q$-subgroup of $G$, of orders $p^m$ and $q^n$, respectively. Since $1\neq Z(S_p)\subseteq C_G(S_p)$, we get $p^{m+1}\mid m_G(S_p)$. Similarly, $q^{n+1}\mid m_G(S_q)$. We infer that $m_G(S_p)\neq m_G(1)$ and $m_G(S_q)\neq m_G(1)$, and so $m_G(S_p)=m_G(S_q)$. Then $p^{m+1}\mid q^n|C_G(S_q)|$, i.e. $p^{m+1}\mid |C_G(S_q)|$, a contradiction.

If there are two distinct primes $p$ and $q$ dividing $|Z(G)|$, then $Z(G)$ contains two subgroups of orders $p$ and $q$, say $H$ and $K$. It results that the following three Chermak-Delgado measures
$$m_G(H)=p|G|,\, m_G(K)=q|G| \mbox{ and } m_G(Z(G))=m_G(G)=|Z(G)||G|$$are distinct, contradicting our hypothesis.

This completes the proof.
\end{proof}

Using Theorem 2.1, we are able to determine the abelian groups in $\calc$.

\begin{corollary}
The cyclic groups of prime order are the unique abelian groups contained in $\calc$.
\end{corollary}

By Corollary 2.2 we easily infer that $\calc$ is not closed under subgroups, homomorphic images, direct products or extensions. Also, from the first part of the proof of Theorem 2.1 we obtain that:

\begin{corollary}
All groups contained in $\calc$ are $p$-groups.
\end{corollary}

Since our study can be reduced to $p$-groups and it is completely finished for abelian groups, in what follows
we will suppose that $G$ is a non-abelian $p$-group of order $p^n$ ($n\geq 3$) belonging to $\calc$. Then:
\begin{itemize}
\item[{\rm a)}] ${\rm Im}(m_G)=\{p^n,p^{n+1}\}$, and consequently $m^*(G)=p^{n+1}$;
\item[{\rm b)}] $Z(G)$ is the unique minimal normal subgroup of $G$, and consequently $Z(G)\subseteq G'\subseteq\Phi(G)$;
\item[{\rm c)}] $HZ(G)\in{\cal CD}(G),\, \forall\, H\leq G$ satisfying $Z(G)\nsubseteq H$.

\hspace{5mm}Indeed, for such a subgroup $H$ of $G$ we have $H\cap Z(G)=1$, and therefore $m_G(H\cap Z(G))=p^n$. Then the inequality
$$m_G(H)m_G(Z(G))\leq m_G(HZ(G))m_G(H\cap Z(G))$$becomes
$$p^{n+1}m_G(H)\leq p^n m_G(HZ(G)),$$that is
$$p\,m_G(H)\leq m_G(HZ(G)).$$Clearly, this implies that $m_G(HZ(G))=p^{n+1}$, i.e. $HZ(G)\in{\cal CD}(G)$.
\end{itemize}

There are many examples of finite non-abelian $p$-groups $G$ such that ${\cal CD}(G)=\{Z(G),G\}$ (see e.g. Corollary 2.2 and Proposition 2.3 of \cite{5}). Using Corollary 2.2 and the above item c), we are able to prove that the intersection between this class of groups and $\calc$ is empty.

\begin{corollary}
$\calc$ does not contain non-abelian $p$-groups $G$ with ${\cal CD}(G)=\{Z(G),G\}$.
\end{corollary}

\begin{proof}
Assume that $\calc$ contains a non-abelian $p$-group $G$ satisfying ${\cal CD}(G)=\{Z(G),G\}$.

If $G$ possesses a minimal subgroup $H\neq Z(G)$, then $HZ(G)\in{\cal CD}(G)$ by c). On the other hand, we obviously have $HZ(G)\neq Z(G)$, and since ${\cal CD}(G)=\{Z(G),G\}$ we get $HZ(G)=G$. Then $|G|=p^2$, implying that $G$ is abelian, a contradiction.

If $Z(G)$ is the unique subgroup of order $p$ in $G$, then $G$ is a generalized quaternion $2$-group by Lemma 1.2, i.e. $p=2$ and $$G\cong Q_{2^n}=\langle a,\,b\mid a^{2^{n-1}}=1,\, a^{2^{n-2}}=b^2,\, b^{-1}ab=a^{-1}\rangle \mbox{ for some } n\geq 3.$$It results that $G$ has a cyclic maximal subgroup $H\cong\langle a\rangle$. So, $$m_G(H)=2^{2n-2}\leq 2^{n+1}=m^*(G),$$which means $n\leq 3$. Since $G$ is non-abelian we get $n=3$, that is $G\cong Q_8$. Then ${\cal CD}(G)$ is a quasi-antichain of width $3$, contradicting the hypothesis.
\end{proof}

Next we will focus on giving examples of non-abelian $p$-groups in $\calc$.

\begin{theorem}
All extraspecial $p$-groups are contained in $\calc$.
\end{theorem}

\begin{proof}
Let $G$ be an extraspecial $p$-group. It is well-known that ${\cal CD}(G)$ consists of all subgroups $H$ of $G$ containing $Z(G)$ (see e.g. Example 2.8 of \cite{8} or Theorem 4.3.4 of \cite{17}). Consequently, all these subgroups have the same Chermak-Delgado measure. On the other hand, by Lemma 2.6 of \cite{2} any subgroup $H$ of $G$ with $Z(G)\nsubseteq H$ satisfies $m_G(H)=|G|$. Thus the function $m_G$ has exactly two values, as desired.
\end{proof}

Using GAP, we are also able to give an example of a non-extraspecial non-abelian $p$-group in $\calc$, namely SmallGroup($32$,$8$):
$$G=\langle a,\,b,\,c \mid a^4=1,\, b^4=a^2,\, c^2=bab^{-1}=a^{-1},\, ac=ca,\, cbc^{-1}=a^{-1}b^3\rangle.$$Note that the nilpotence class of $G$ is $3$. Also, ${\cal CD}(G)$ is described in Lemma 4.5.16 and Corollaries 4.5.20 and 4.5.21 of \cite{17}.
\bigskip

We observe that all non-abelian groups of order $p^3$ belongs to $\calc$ because they are extraspecial. The same thing cannot be said about non-abelian groups of order $p^4$: by Lemma 1.1 such a group $G$ has an abelian subgroup $A$ of order $p^3$, and so $m^*(G)\geq m_G(A)=p^6>p^5$, implying that $G$ is not contained in $\calc$. This argument can be extended in the following way.

\begin{proposition}
If a non-abelian group of order $p^n$ contains an abelian subgroup of order $\geq p^{\,\left[\frac{n+3}{2}\right]}$, then it does not belong to $\calc$.
\end{proposition}

Since any group of order $64$ contains an abelian subgroup of order $16$, by Proposition 2.6 we infer that:

\begin{corollary}
$\calc$ does not contain non-abelian groups of order $64$.
\end{corollary}

Another application of Proposition 2.6 is the following:

\begin{theorem}
Let $G$ be a finite $p$-group of nilpotence class $2$ contained in $\calc$. Then $G$ is extraspecial.
\end{theorem}

\begin{proof}
Since the nilpotence class of $G$ is $2$, we have that $G/Z(G)$ is abelian and so $G'\subseteq Z(G)$. By Theorem 1 we get $G'=Z(G)$, which implies that $G$ is an outer abelian $p$-group. Then $G$ belongs to one of the four classes of groups in Lemma 1.3.

We observe that $M(n,1)$ has a cyclic subgroup of order $p^n$, namely $\langle a\rangle$, and $n\geq \left[\frac{n+4}{2}\right]$ for $n\geq 3$. Thus it cannot be contained in $\calc$ by Proposition 2.6. Also, it is easy to see that a central product $E*A$, where $E$ is an extraspecial $p$-group of order $p^{2m+1}$ and $A\cong M(n,1)$, $n\geq 3$, always has an abelian subgroup of order $p^{m+n}$. Since $m+n\geq \left[\frac{2m+n+4}{2}\right]$ for $n\geq 3$, by Proposition 2.6 we infer that $E*A$ does not belong to $\calc$. Similarly, a central product $E*A$, where $E$ is an extraspecial $p$-group of order $p^{2m+1}$ and $A\cong C_{p^t}$ with $t\geq 2$, always has an abelian subgroup of order $p^{m+t}$. If $t\geq 3$ then $m+t\geq \left[\frac{2m+t+4}{2}\right]$, implying that $E*A$ is not contained in $\calc$. If $t=2$, it suffices to observe that the center of $E*A$ is of order $p^2$, and consequently $E*A$ is not contained in $\calc$ by Theorem 1. These shows that the unique possibility is that $G$ be an extraspecial $p$-group, as desired.
\end{proof}

Our last result shows that the non-abelian groups of order $p^3$ are in fact the unique $p$-groups of maximal class in $\calc$.

\begin{theorem}
Let $G$ be a finite $p$-group of maximal class contained in $\calc$. Then $G$ is non-abelian of order $p^3$.
\end{theorem}

\begin{proof}
Obviously, $G$ is non-abelian. Let $|G|=p^n$. By Lemma 1.4 we know that $G$ possesses a subgroup $A$ of order $p^2$ such that $C_G(A)=A$. It follows that $m_G(A)=p^4$, and therefore we have either $n=3$ or $n=4$. Since the case $n=4$ is impossible, we get $n=3$, as desired.
\end{proof}

Inspired by the above examples, we end this note by indicating the following open problem.

\bigskip
\noindent{\bf Open problem.} Which are the pairs $(p,n)$, where $p$ is a prime and $n$ is a positive integer, such that $\calc$ contains groups of order $p^n$?

Note that all pairs $(p,n)$ with $n$ odd satisfy this property by Corollary 2.2 and Theorem 2.5.

\vspace*{3ex}\small

\hfill
\begin{minipage}[t]{5cm}
Marius T\u arn\u auceanu \\
Faculty of  Mathematics \\
``Al.I. Cuza'' University \\
Ia\c si, Romania \\
e-mail: {\tt tarnauc@uaic.ro}
\end{minipage}


\begin{thebibliography}{10}
\bibitem{1} Y. Berkovich, {\it Groups of prime power order}, vol. 1, de Gruyter, Berlin, 2008.
\bibitem{2} S. Bouc and N. Mazza, {\it The Dade group of {\rm(}almost{\rm)} extraspecial $p$-groups}, J. Pure Appl. Algebra {\bf 192} (2004), 21-51.
\bibitem{3} L. An, J.P. Brennan, H. Qu and E. Wilcox, {\it Chermak-Delgado lattice extension theorems}, Comm. Algebra {\bf 43} (2015), 2201-2213.
\bibitem{4} B. Brewster and E. Wilcox, {\it Some groups with computable Chermak-Delgado lattices}, Bull. Aus. Math. Soc. {\bf 86} (2012), 29-40.
\bibitem{5} B. Brewster, P. Hauck and E. Wilcox, {\it Groups whose Chermak-Delgado lattice is a chain}, J. Group Theory {\bf 17} (2014), 253-279.
\bibitem{6} B. Brewster, P. Hauck and E. Wilcox, {\it Quasi-antichain Chermak-Delgado lattices of finite groups}, Archiv der Mathematik {\bf 103} (2014), 301-311.
\bibitem{7} A. Chermak and A. Delgado, {\it A measuring argument for finite groups}, Proc. AMS {\bf 107} (1989), 907-914.
\bibitem{8} G. Glauberman, {\it Centrally large subgroups of finite $p$-groups}, J. Algebra {\bf 300} (2006), 480-508.\newpage
\bibitem{9} I.M. Isaacs, {\it Finite group theory}, Amer. Math. Soc., Providence, R.I., 2008.
\bibitem{10} R. McCulloch, {\it Chermak-Delgado simple groups}, Comm. Algebra {\bf 45} (2017), 983-991.
\bibitem{11} R. McCulloch, {\it Finite groups with a trivial Chermak-Delgado subgroup}, J. Group Theory {\bf 21} (2018), 449-461.
\bibitem{12} R. McCulloch and M. T\u arn\u auceanu, {\it Two classes of finite groups whose Chermak-Delgado lattice is a chain of length zero}, Comm. Algebra {\bf 46} (2018), 3092-3096.
\bibitem{13} R. McCulloch and M. T\u arn\u auceanu, {\it On the Chermak-Delgado lattice of a finite group}, submitted.
\bibitem{14} M. Suzuki, {\it Group theory}, I, II, Springer Verlag, Berlin, 1982, 1986.
\bibitem{15} M. T\u arn\u auceanu, {\it The Chermak-Delgado lattice of ZM-groups}, Results Math. {\bf 72} (2017), 1849-1855.
\bibitem{16} M. T\u arn\u auceanu, {\it A note on the Chermak-Delgado lattice of a finite group}, Comm. Algebra {\bf 46} (2018), 201-204.
\bibitem{17} L.S. Vieira, {\it On $p$-adic fields and $p$-groups}, Ph.D. Thesis, University of Kentucky, 2017.
\bibitem{18} E. Wilcox, {\it Exploring the Chermak-Delgado lattice}, Math. Magazine {\bf 89} (2016), 38-44.
\bibitem{19} Q. Zhang, L. Li and M. Xu, {\it Finite $p$-groups all of whose proper quotient groups are abelian of inner-abelian}, Comm. Algebra {\bf 38} (2010), 2797-2807.
\bibitem{20} A. Morresi Zuccari, V. Russo, and C.M. Scoppola, {\it The Chermak-Delgado measure in finite $p$-groups}, J. Algebra {\bf 502} (2018), 262-276.
\end{thebibliography}
\end{document}